\documentclass[12pt]{article}
\usepackage{amssymb,amsmath,amscd,epsfig,amsthm}
\usepackage{dsfont}
\usepackage{hyperref}
\usepackage{fullpage} 

\newcommand{\G}{\mathcal G}
\newcommand{\A}{\mathcal A}

\newcommand{\Z}{\mathbb Z}
\newcommand{\Sym}{\mathop{\rm Sym}\nolimits}

\newcommand{\Xd}{X_{\diamond}}
\newcommand{\Cay}{\mathop{\rm Cay}\nolimits}

\newtheorem{theorem}{Theorem}[section]

\newtheorem{lemma}[theorem]{Lemma}
\newtheorem{definition}{Definition}
\newtheorem{question}{Question}

\begin{document}

\title{An Example of an Automatic Graph of Intermediate Growth}

\author{Alexei Miasnikov\footnote{Partially Supported by the Marsden Fund of The Royal Society of New Zealand.}\\
               Department of Mathematical Sciences,\\
               Stevens Institute of Technology,\\
               Castle Point, Hoboken, NJ, 07030\\
               \href{mailto:amyasnik@stevens.edu}{amyasnik@stevens.edu}\\
               \and
        Dmytro Savchuk\footnote{Partially Supported by the New Researcher Grant from University of South Florida.}\\
               Department of Mathematics and Statistics\\
               University of South Florida\\
               4202 E Fowler Ave\\
               Tampa, FL 33620-5700\\
               \href{mailto:savchuk@usf.edu}{savchuk@usf.edu}
}

\maketitle

\begin{abstract}
We give an example of a 4-regular infinite automatic graph of intermediate growth. It is constructed as a Schreier graph of a certain group generated by 3-state automaton. The question was motivated by an open problem on the existence of Cayley automatic groups of intermediate growth.
\end{abstract}


\section*{Introduction}
Automatic groups were formally introduced by Thurston in 1986 motivated by earlier results of Cannon~\cite{cannon:combinatorial_structure84} on properties of Cayley graphs of hyperbolic groups. The latter results, in turn, were motivated by the pioneering work of Dehn on word problem in surface groups. All automatic groups have solvable in a quadratic time word problem and have at most quadratic Dehn function. If, in addition, a group is bi-automatic, then it has solvable conjugacy problem. For survey on the main results about the class of automatic groups we refer the reader to the multi-author book~\cite{epstein_chpt:word_processing_in_groups92}.

However, the class of automatic groups has its limitations. First of all, many of groups that play an important role in geometric group theory are not automatic. These include finitely generated nilpotent groups that are not virtually abelian, Baumslag-Solitar groups $BS(p,q)$ (unless $p=0$, $q=0$, or $p=\pm q$), non-finitely presented groups, infinite torsion groups, $SL_n(\Z)$. This would be desirable to extend the class of automatic groups to some wider class while preserving the computational routines of automatic groups. Further, some of the very basic questions about the class of automatic groups have still not been solved despite considerable efforts by the mathematical community. For example, it is not known whether each automatic group is bi-automatic.

In view of the above arguments it was quite natural to search for possible generalizations of the class of automatic groups. Several papers offered different approaches. Combable groups share with automatic groups the fellow traveler property, but have weaker constraints on the language used in the definition. Bridson in~\cite{bridson:combings03} discusses the relation between these two classes. The geometric generalization of the class of automatic groups, so-called, asynchronously indexed-combable groups, was defined and studied by  Gillman and Bridson in~\cite{bridson_g:formal_language_theory96}. It uses indexed languages and covers the fundamental groups of all compact 3-manifold  satisfying the geometrization conjecture. Unfortunately, this class looses certain important algorithmic features of automatic groups. Recently Brittenham and Hermiller~\cite{brittenham_h:stackable} defined another related class of stackable groups. They show, in particular, that every shortlex automatic group, including every word hyperbolic group, is regularly stackable, and that each stackable group is finitely presented. The exact relationship between these two classes is not yet fully understood.

The notion of a Cayley automatic group was introduced and studied  in~\cite{kharlampovich_km:automatic_structures14} as a natural generalization of the class of automatic groups. It has been observed that the Cayley graphs of automatic groups are automatic with respect to special encoding, in the sense of the theory of automatic structures developed, in particular, by Khoussainov and Nerode~\cite{khoussainov_n:automatic_presentations95}. This theory can be traced back to works of Hodgson in the end of 1970's -- beginning of 1980's~\cite{hodgson:decidabilite_par_automate_fini83}. For a survey on the results in this theory we refer the reader to a paper by Rubin~\cite{rubin:automata_presenting_structures_survey08}. A natural way to generalize the notion of automatic groups would be to remove the condition on the encoding on Cayley graphs. In other words, a group is called Cayley automatic, if its Cayley graph is automatic.

The class of Cayley automatic groups retains many algorithmic properties of the class of automatic groups, but is much wider. In particular, it includes many examples of nilpotent and solvable groups, which are not automatic in the standard sense. Some of Cayley automatic groups are not finitely presented. For example, the restricted wreath product of a nontrivial finite group $G$ by $\Z$ is Cayley automatic. Further, it was recently shown by Miasnikov and \v Suni\'c in~\cite{miasnikov_s:cayley_automatic11} that there exist Cayley automatic groups that are not Cayley biautomatic, thus resolving an analogue of a longstanding question of the theory of automatic groups. At the same time, main algorithmic tools of automatic groups still work. In particular, the word problem in each Cayley automatic group can be decided in a quadratic time.

Even further generalization of Cayley automatic groups, was recently introduced and studied by Elder and Taback in~\cite{elder_t:Cgraph_automatic_groups}. For each class of languages $\mathcal C$ they define $\mathcal C$-graph automatic groups in exactly the same way as Cayley automatic groups with the difference that the formal languages used in the definition must belong to class $\mathcal C$. In particular, if $\mathcal C$ is the class of regular languages, one simply obtains the class of Cayley automatic groups. One of the motivations to consider other classes of languages is the fact proved in~\cite{elder_t:Cgraph_automatic_groups} that polynomial time word problem algorithm is still preserved if one replaces the class of regular languages by the class of counter languages.

This paper was motivated by the following natural question regarding possible limitations of the class of Cayley automatic groups.

\begin{question}
Is there a Cayley Automatic group of intermediate growth?
\end{question}

Recall that the growth function of a finitely generated group $G$ with respect to a generating set $S$ is a function $\gamma_{G,S}\colon \mathbb N\to\mathbb N$ such that $\gamma_{G,S}(n)$ is equal to the number of elements of $G$ that can be expressed as a product of at most $n$ elements of $S\cup S^{-1}$. More generally, the growth function $\gamma_{\Gamma,x}(n)$ of a locally finite graph $\Gamma$ with respect to the selected base point $x$ is a function such that $\gamma_{\Gamma,x}(n)$ is the number of elements in the ball of radius $n$ in $\Gamma$ centered at the base point. The growth function $\gamma_{G,S}$ can then be defined as $\gamma_{\Cay(G,S),e}$, where $\Cay(G,S)$ is the Cayley graph of $G$ with respect to a generating set $S$ and $e$ is the identity element in $G$. The growth function of any finitely generated group cannot grow faster than the exponential function and, according to Gromov's celebrated theorem, $\gamma_{G,S}(n)$ grows as a polynomial function if and only if $G$ is virtually nilpotent. It was a longstanding question posed by Milnor if there is a group whose growth function is intermediate, i.e. it grows faster than any polynomial function, but slower than the exponential function~\cite{milnor:problem}. The first example of such group was constructed by Grigorchuk in~\cite{grigorch:degrees}, but even up to now all known constructions of such groups are based to certain extent on ideas from the original construction in~\cite{grigorch:degrees}. It is known that automatic groups cannot have intermediate growth, which is a limitation that might be overcame by passing to a bigger class. Also, there are no known examples of finitely presented groups of intermediate growth, which makes the class of Cayley automatic groups particularly appealing. Unlike the class of automatic groups it contains many groups that are not finitely presented.

By the definition, a group is Cayley automatic if and only if its Cayley graph is automatic (i.e. admits automatic structure). The main purpose of this paper is the following theorem.

\begin{theorem}
There is an automatic graph of intermediate growth.
\end{theorem}

The graph that we use to prove the main theorem belongs to the family of graphs of intermediate growth constructed in~\cite{bond_cdn:amenable} as the family of Schreier graphs of the action of a group generated by the two nontrivial states of the 3-state automaton depicted in Figure~\ref{fig:929} on the boundary of a binary rooted tree. One of the graphs in this family was constructed earlier by Benjamini and Hoffman in~\cite{benjamini_h:omega_per_graphs}. The last paper also gives credit to Bartholdi who pointed out that the graph under consideration was, in fact, a Schreier graph of a group generated by automaton.

The paper is organized as follows. In Section~\ref{sec:prelim} we recall the main definitions related to groups generated by automata. The main example of a graph of intermediate growth $\Gamma_{(01)^\infty}$ and the group $G$ acting on this graph is given in Section~\ref{sec:schreier}. Section~\ref{sec:CayleyAutom} introduces the notions of an automatic graph and of a Cayley automatic group. Finally, Section~\ref{sec:main} contains the proof that the graph $\Gamma_{(01)^\infty}$ is automatic.

\noindent{\bf Acknowledgement.} The authors are grateful to Thomas Colcombet for useful discussions and to the anonymous referee whose valuable suggestions have enhanced the exposition of the paper and optimized some proofs.

\section{Groups generated by automata}
\label{sec:prelim}

Let $X$ be a finite set of cardinality $d$ and let $X^*$ denote the set of all finite words over $X$ (that can be though as the free monoid generated by $X$). This set can be naturally endowed with a structure of a
rooted $d$-ary tree by declaring that $v$ is adjacent to $vx$ for
any $v\in X^*$ and $x\in X$. The empty word corresponds to the root
of the tree and $X^n$ corresponds to the $n$-th level of the tree.
We will be interested in the groups of graph automorphisms and semigroups
of graph homomorphisms of $X^*$. Any such homomorphism can be defined via
the notion of initial automaton.

\begin{definition}
A \emph{Mealy automaton} (or simply \emph{automaton}) is a tuple
$(Q,X,\pi,\lambda)$, where $Q$ is a set (a set of states), $X$ is a
finite alphabet, $\pi\colon Q\times X\to Q$ is a transition function
and $\lambda\colon Q\times X\to X$ is an output function. If the set
of states $Q$ is finite the automaton is called \emph{finite}. If
for every state $q\in Q$ the output function $\lambda(q,x)$ induces
a permutation of $X$, the automaton $\A$ is called invertible.
Selecting a state $q\in Q$ produces an \emph{initial automaton}
$\A_q$.
\end{definition}

Automata are often represented by the \emph{Moore diagrams}. The
Moore diagram of an automaton $\A=(Q,X,\pi,\lambda)$ is a directed
graph in which the vertices are the states from $Q$ and the edges
have form $q\stackrel{x|\lambda(q,x)}{\longrightarrow}\pi(q,x)$ for
$q\in Q$ and $x\in X$. If the automaton is invertible, then it is
common to label vertices of the Moore diagram by the permutation
$\lambda(q,\cdot)$ and leave just first components from the labels
of the edges. An example of Moore diagram is shown in
Figure~\ref{fig:929}.

Any initial automaton induces a homomorphism of $X^*$. Given a word
$v=x_1x_2x_3\ldots x_n\in X^*$ it scans its first letter $x_1$ and
outputs $\lambda(x_1)$. The rest of the word is handled in a similar
fashion by the initial automaton $\A_{\pi(x_1)}$. Formally speaking,
the functions $\pi$ and $\lambda$ can be extended to $\pi\colon
Q\times X^*\to Q$ and $\lambda\colon  Q\times X^*\to X^*$ via
\[\begin{array}{l}
\pi(q,x_1x_2\ldots x_n)=\pi(\pi(q,x_1),x_2x_3\ldots x_n),\\
\lambda(q,x_1x_2\ldots x_n)=\lambda(q,x_1)\lambda(\pi(q,x_1),x_2x_3\ldots x_n).\\
\end{array}
\]

By construction any initial automaton acts on $X^*$ as a
homomorphism and every invertible initial automaton acts on $X^*$ as an
automorphism.

\begin{definition}
The semigroup (group) generated by all states of an automaton $\A$ is
called an \emph{automaton semigroup} (\emph{automaton group}) and
denoted by $\mathds S(\A)$ (respectively $\mathds G(\A)$).
\end{definition}

Another popular name for automaton groups and semigroups is
self-similar groups and semigroups
(see~\cite{nekrash:self-similar}).

Conversely, any homomorphism of $X^*$ can be encoded by the action
of an initial automaton. In order to show this we need a notion of a
\emph{section} of a homomorphism at a vertex of the tree. Let $g$ be
a homomorphism of the tree $X^*$ and $x\in X$. Then for any $v\in
X^*$ we have
\[g(xv)=g(x)v'\]
for some $v'\in X^*$. Then the map $g|_x\colon X^*\to X^*$ given by
\[g|_x(v)=v'\]
defines a homomorphism of $X^*$ and is called the \emph{section} of
$g$ at vertex $x$. Furthermore,  for any $x_1x_2\ldots x_n\in X^*$
we define \[g|_{x_1x_2\ldots x_n}=g|_{x_1}|_{x_2}\ldots|_{x_n}.\]

Given a homomorphism $g$ of $X^*$ we construct an initial automaton
$\A(g)$ whose action on $X^*$ coincides with that of $g$ as follows.
The set of states of $\A(g)$ is the set $\{g|_v\colon  v\in X^*\}$
of different sections of $g$ at the vertices of the tree. The
transition and output functions are defined by
\[\begin{array}{l}
\pi(g|_v,x)=g|_{vx},\\
\lambda(g|_v,x)=g|_v(x).
\end{array}\]

Throughout the paper we will use the following convention. If $g$
and $h$ are the elements of some (semi)group acting on set $A$ and
$a\in A$, then
\begin{equation}
\label{eqn_conv}
gh(a)=h(g(a)).
\end{equation}

Taking into account convention~\eqref{eqn_conv} one can compute
sections of any element of an automaton semigroup as follows. If
$g=g_1g_2\cdots g_n$ and $v\in X^*$, then

\begin{equation}
\label{eqn_sections} g|_v=g_1|_v\cdot g_2|_{g_1(v)}\cdots
g_n|_{g_1g_2\cdots g_{n-1}(v)}.
\end{equation}

For any automaton group $G$ there is a natural embedding
\[G\hookrightarrow G \wr \Sym(X)\]
defined by
\[G\ni g\mapsto (g_1,g_2,\ldots,g_d)\lambda(g)\in G\wr \Sym(X),\]
where $g_1,g_2,\ldots,g_d$ are the sections of $g$ at the vertices
of the first level, and $\lambda(g)$ is a permutation of $X$ induced by the action of $g$ on the first level of the tree.

The above embedding is convenient in computations involving the
sections of automorphisms, as well as for defining automaton groups. Sometimes it is called the \emph{wreath recursion} defining the group.

Finally, we note that any homomorphism of $X^*$ induces an action on the set $X^\infty$ of all infinite words over $X$ that can be viewed as a boundary of the tree $X^*$.

\section{Definition of the group and structure of the graph}
\label{sec:schreier}
The main graph studied in this paper is a Schreier graph of a certain group $G$ defined below. We start this section from recalling the definition of a Schreier graph.

\begin{definition}
Let $G$ be a group generated by a finite generating set $S$ acting on a set $M$. The \emph{(orbital) Schreier graph} $\Gamma(G,S,M)$ of the action of $G$ on $M$ with respect to the generating set $S$ is an oriented labeled graph defined as follows. The set of vertices of $\Gamma(G,S,M)$ is $M$ and there is an arrow from $x\in M$ to $y\in M$ labeled by $s\in S$ if and only if $x^s=y$, where $x^s$ denotes the image of $x$ under the action of $s$. We will call a Schreier graph with a selected basepoint a \emph{pointed} Schreier graph.
\end{definition}

An equivalent view on Schreier graphs goes back to Schreier, who called these graphs \emph{coset graphs}. For any subgroup $H$ of $G$, the group $G$ acts on the right $H$-cosets $G/H$ by right multiplication. This action gives rise to the Schreier graph $\Gamma(G,S,G/H)$. Conversely, if $G$ acts on $M$ transitively, then $\Gamma(G,S,M)$ is canonically isomorphic to $\Gamma(G,S,G/\mathop{\rm Stab}\nolimits_G(x))$ for any $x\in M$, where the vertex $y\in M$ in $\Gamma(G,S,M)$ corresponds to the coset from $G/\mathop{\rm Stab}\nolimits_G(x)$ consisting of all elements of  $G$ that move $x$ to $y$.  Also, to simplify notation, we will refer to $\Gamma(G,S,G/\mathop{\rm Stab}\nolimits_G(x))$ as the Schreier graph of $x$ and denote it by $\Gamma_x$ when the group, the set, and the action are clear from the context.

Consider a group $\G$ generated by two nontrivial states of a 3-state automaton $\mathcal A$ over 2-letter alphabet $X=\{0,1\}$ defined by the following wreath recursion
\[\begin{array}{l}
a=(e,a)\sigma,\\
b=(b,a),
\end{array}\]
where $e$ denotes the identity of $\G$ and $\sigma$ is a nontrivial permutation of $\{0,1\}$. The Moore diagram of this automaton is shown in Figure~\ref{fig:929}.

\begin{figure}
\begin{center}
\epsfig{file=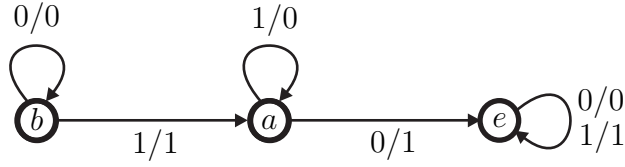}
\end{center}
\caption{Automaton generating group $\G$\label{fig:929}}
\end{figure}

This group acts on the boundary $\{0,1\}^\infty$ of a tree $\{0,1\}^*$ and this action induces an uncountable family of pointed orbital Schreier graphs $\Gamma_\omega$ for each $\omega\in \{0,1\}^\infty$. Namely, $\Gamma_\omega$ is an orbital Schreier graph of the action of $\G$ on the orbit of $\omega$ with respect to the generating set $S=\{a,b\}$ with the basepoint $\omega$. This family of graphs was completely described in~\cite{bond_cdn:amenable} and we borrow our notation from this paper.

The structure of $\Gamma_\omega$ is as follows. The vertices of $\Gamma_\omega$ are identified with integers and the set of edges $E_\omega$ consists of countably many families $E_\omega^n$, $n\geq0$ and, possibly, one loop based at one of the vertices of $\Gamma_\omega$. The family $E_\omega^0$ is defined as
\[E_\omega^0=\{(n,n+1)\colon n\in\Z\}.\]
Each successive $E_\omega^n$, $n>0$ is defined recursively. If $\omega=x_1x_2x_3\ldots$ and $x_n=0$, then let $c_n^\omega$ be the largest nonpositive integer that is not the endpoint of any of the edges in $E_\omega^1,\ldots,E_\omega^{n-1}$. If $x_n=1$, then let $c_n^\omega$ be the smallest positive integer that is not the endpoint of any of the edges in $E_\omega^1,\ldots,E_\omega^{n-1}$.
The family $E_\omega^n$ is now defined as
\[E_\omega^n=\{(2^nz+c_n^\omega, 2^n(z+1)+c_n^\omega)\colon z\in\Z\}.\]
By construction, if there are both infinitely many 0's and infinitely many 1's in $\omega$, then each vertex in $\Gamma_\omega$ will be adjacent to exactly $4$ edges in $\cup_{n\geq0}E_\omega^n$. In this case we simply have
\[E_\omega=\cup_{n\geq0}E_\omega^n.\]
If there is only a finite number of $0$'s or $1$'s in $\omega$, then all vertices in $\Gamma_\omega$ except exactly one vertex $t$ will have four adjacent edges in $\cup_{n\geq0}E_\omega^n$, while $t$ will be an endpoint of only two edges from $E_\omega^0$. In this case
\[E_\omega=\{\text{loop at $t$}\}\cup(\cup_{n\geq0}E_\omega^n).\]

In particular, graph $\Gamma_{(01)^\infty}$ is shown in Figure~\ref{fig:gamma0101}.

\begin{figure}
\begin{center}
\epsfig{file=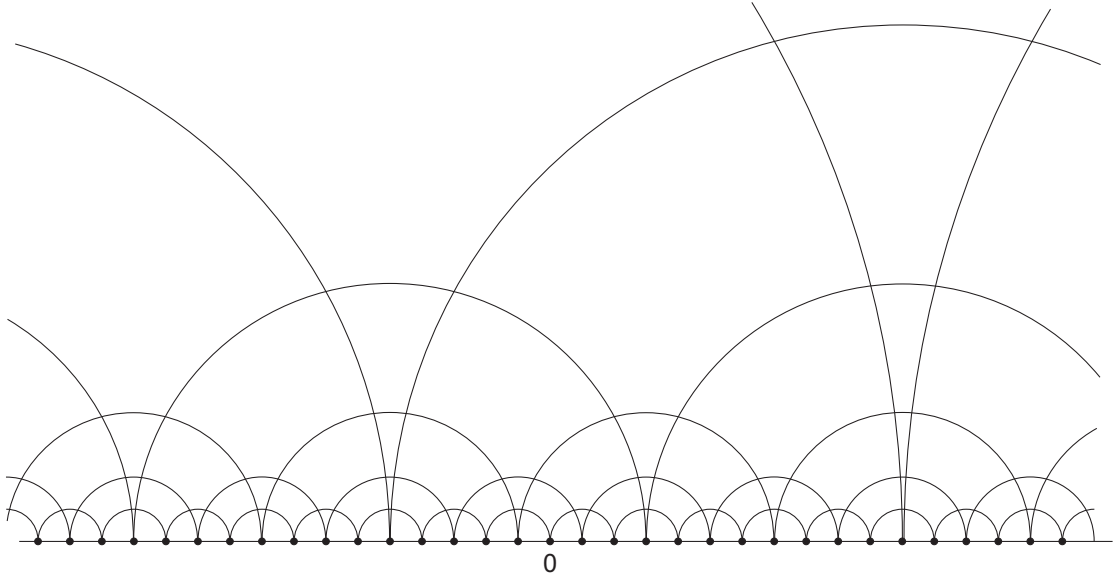,width=420pt}
\end{center}
\caption{Graph $\Gamma_{(01)^\infty}$\label{fig:gamma0101}}
\end{figure}

The following theorem has been proved in~\cite{bond_cdn:amenable}:

\begin{theorem}[\cite{bond_cdn:amenable}]
All orbital Schreier graphs $\Gamma_\omega$ for $\omega\in\{0,1\}^\infty$ of the group $\G$ have intermediate growth. More specifically, the growth function satisfies
\[n^{\frac12log_2n}\preceq |B(\omega,n)|\preceq n^{log_2n}\]
\end{theorem}

The above theorem is a generalization of an earlier result of Benjamini and Hoffman~\cite{benjamini_h:omega_per_graphs} who, in particular, proved that $\Gamma_{0^\infty}$ has intermediate growth.

\section{Automatic Graphs and Cayley Automatic Groups}
\label{sec:CayleyAutom}
Let $X$ by a finite alphabet. For a special symbol $\diamond\notin X$ we define an extended alphabet $X_\diamond=X\cup\{\diamond\}$. For a pair $(w_1,w_2)$ of finite words over $X$ we define a \emph{convolution} or a \emph{padded pair} (see, for example,~\cite{holt:automatic_groups98}) $\otimes(w_1,w_2)$ to be the word over $(\Xd)^2$ of length $\max\{|w_1|,|w_2|\}$, whose $j$-th symbol is $(\sigma_1,\sigma_2)$, where
\[\sigma_i=\left\{\begin{array}{ll}
\text{the $j$-th symbol of $w_i$}, &\text{if\ } j\leq |w_i|\\
\diamond, & \text{otherwise}
\end{array}\right.\]
For example, if $X=\{0,1\}$, then
\[\otimes(011,00110)=\left(\!\!\!\begin{array}{c}0\\0\end{array}\!\!\!\right)\left(\!\!\!\begin{array}{c}1\\0\end{array}\!\!\!\right)\left(\!\!\!\begin{array}{c}1\\1\end{array}\!\!\!\right)\left(\!\!\!\begin{array}{c}\diamond\\1\end{array}\!\!\!\right)\left(\!\!\!\begin{array}{c}\diamond\\0\end{array}\!\!\!\right),\]
where letters of $(\Xd)^2$ are written for convenience as columns. We note that $w_1$ and $w_2$ can be empty.

Let $R$ be a binary relation on $X^*$. The \emph{convolution} of $R$ is the language over $(\Xd)^2$ defined by
\[\otimes R=\{\otimes(w_1,w_2)\ |\ (w_1,w_2)\in R\}.\]

A binary relation $R$ on $X^*$ is called \emph{regular} if its convolution $\otimes R$ is a regular language over $(\Xd)^2$, i.e. $\otimes R$ is recognizable by a finite automaton acceptor over $(\Xd)^2$. To avoid possible confusion we emphasize that the automata acceptors here are different from automata transducers defined in Section~\ref{sec:prelim}.

Now we proceed to the definition of automatic graphs and Cayley automatic groups.

Let $\Gamma=(V,E,\sigma\colon E\to S)$ be graph whose edges are labeled by elements of finite set $S=\{s_1,s_2,\ldots,s_n\}$ according to the map $\sigma$. This graph can be interpreted as a system of $|S|$ binary relations $E_s$ on $V$, for $s\in S$, where
\[E_s=\{(v,v')\ |\ (v,v')\in E\ \text{and the label of $(v,v')$ is $s$}\}.\]

Each map $\overline{\phantom{V}}\colon V\to X^*$ induces $|S|$ binary relations $\overline{E}_s$ on $X^*$ given by
\[\overline{E}_s=\{(\overline v,\overline{v'})\ |\ (v,v')\in E_s\}.\]

The definition of automatic graph below is a particular instance of an automatic structure~\cite{khoussainov_n:automatic_presentations95,kharlampovich_km:automatic_structures14}.

\begin{definition}
The labeled graph $\Gamma=(V,E,\sigma\colon E\to S)$ is called \emph{automatic}, if there is a finite alphabet $X$ and an injective map $\overline{\phantom{V}}\colon V\to X^*$ such that
\begin{itemize}
\item $\overline V$ is a regular language over $X$ and
\item $\overline{E}_s$ is a regular binary relation on $X^*$ for each $s\in S$.
\end{itemize}
In such a case, the tuple $(\overline V,\overline{E}_{s_1},\overline{E}_{s_2},\ldots,\overline{E}_{s_k})$ is called an \emph{automatic structure} on graph $\Gamma$ with respect to $S=\{s_1,s_2,\ldots,s_k\}$.
\end{definition}

A rich source of examples of labeled graphs comes from group actions. Let $G$ be a finitely generated group with a finite generating set $S$. The (right) Schreier graph $\Gamma(G,S,Y)$ of the action of $G$ on $Y$ is a graph, whose vertex set is $Y$, and for each $y\in Y$ and $s\in S$ there is  an edge labelled by $s$ from $y$ to $y^s$. The (right) Cayley graph of $G$ can be thought of as a Schreier graph of the regular action of $G$ on itself by multiplication on right.

The definition of Cayley automatic groups from~\cite{kharlampovich_km:automatic_structures14} is as follows:

\begin{definition}
A finitely generated group $G$ with finite generating set $S$ is \emph{Cayley automatic} if its Cayley graph $\Cay(G,S)$ with respect to $S$ is automatic.
\end{definition}

We note that even though the property of being Cayley automatic depends only on a group, and not on a finite generating set (see~\cite{kharlampovich_km:automatic_structures14}), the same group can have both automatic and non-automatic Schreier graphs. For example, a group which is not Cayley automatic certainly acts trivially on the one element set producing an automatic Schreier graph.

\section{Main Result}
\label{sec:main}

It is an open question whether there is a Cayley automatic group of intermediate growth. We do not answer this question here, however, we construct an automatic Schreier graph of intermediate growth. The main purpose of this note is the following theorem.

\begin{theorem}
\label{thm:main}
The Schreier graph $\Gamma_{(01)^\infty}$ of intermediate growth is automatic.
\end{theorem}

The proof of this theorem will be elaborated through the lemmas below. Throughout the proof we will denote $\Gamma_{(01)^\infty}$ simply by $\Gamma$. First, we produce an injection $\overline{\phantom{G}}\colon V(\Gamma)\to X^*$. This amounts to labelling the vertices of $\Gamma$ by different words over $X$. Since $\Gamma$ is a Schreier graph of the action of $\G$ on the orbit of $(01)^\infty$, the vertices of $\Gamma$ are already labelled by infinite words over $X$.

Recall that two infinite words $\omega=x_1x_2x_3\ldots$ and $\omega'=y_1y_2y_3\ldots$ in $X^\infty$ are called \emph{cofinal} if there exist $N>0$ such that $x_n=y_n$ for all $n\geq N$. It is proved in~\cite{bond_cdn:amenable} that the orbit of $\omega\in X^\infty$ coincides with the cofinality class of $\omega$ except the case if $\omega$ is cofinal to $0^\infty$ or $1^\infty$, when the orbit coincides with the union of cofinality classes of $0^\infty$ and $1^\infty$. Therefore, in the case of $\omega=(01)^\infty$, each vertex of $\Gamma$ is initially labelled by an infinite word over $X$ that is cofinal with $(01)^\infty$. We define an injection $\overline{\phantom{G}}\colon V(\Gamma)\to X^*$ by sending each vertex $v$ to the prefix of its label of length $l$ with the property that $l$ is the largest nonnegative integer such that the $l$-th digit of the label of $v$ differs from the $l$-th digit in $(01)^\infty$. So, for example, we have
\[\begin{array}{l}
\overline{0110(01)^\infty}=0110,\\
\overline{(01)^\infty}=\emptyset,
\end{array}\]
where $\emptyset$ denotes the empty word over $X$.

\begin{lemma}
\label{lem:vertices}
The set $\overline{V(\Gamma)}\subset X^*$ is a regular language over $X$.
\end{lemma}

\begin{proof}
First of all, we observe that $w=x_1x_2\ldots x_l\in\overline{V(\Gamma)}$ if and only if either $w=\emptyset$ or the last letter $x_l$ of $w$ is different from the $l$-th letter of $(01)^\infty$.

\begin{figure}
\begin{center}
\epsfig{file=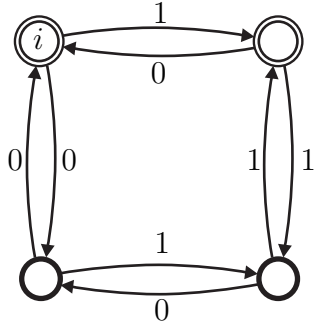}
\end{center}
\caption{Automaton $\A_V$ accepting $\overline{V(\Gamma_{(01)^\infty})}$\label{fig:aut_vertices}}
\end{figure}

It is straightforward to verify now that $\overline{V(\Gamma)}$ is accepted by the automaton $\A_V$ depicted in Figure~\ref{fig:aut_vertices}, where the initial state is labelled by $i$ and the terminal states are marked by double circles. Indeed, when automaton reads word $w=x_1x_2\ldots x_l$ over $X$ starting from the initial (top left) state, we can keep track of whether the last letter of $w$ is different from the $l$-th lettter of $(01)^\infty$ by looking at the state in which we end up after reading $w$. If we end up in one of the top two states, then $w$ ends with the letter opposite to the $l$-th letter in $(01)^\infty$ and is accepted by $\A_V$. On the contrary, if we end up in one of the bottom two states, then $w$ ends with the $l$-th letter in $(01)^\infty$ and is not accepted by $\A_V$.
\end{proof}

Note that more generally, we can similarly define $\overline{\phantom{G}}\colon V(\Gamma_{\omega})\to X^*$ for any $\omega\in X^\infty$. In the case of preperiodic $\omega$ the analog of Lemma~\ref{lem:vertices} can be proved by constructing a similar automaton. But we will not need this more general result here.

Before proving that $\overline{E}_a$ and $\overline{E}_b$ are regular relations on $X^*$ we prove the following auxiliary lemma.

\begin{lemma}
\label{lemma:pairs}
For every regular language $L$ over $X$ the language $L_{pairs}=\{\otimes(u,v)\ |\ u,v\in L\}$ is regular over $(\Xd)^2$.
\end{lemma}

\begin{proof}
There is a natural one-to-one correspondence between finite words over $(\Xd)^2$ and pairs of words over $\Xd$ of the same length. Thus, for words $u=x_1x_2\ldots x_n$ and $v=y_1y_2\ldots y_n$ over $\Xd$ of the same length we will sometimes denote by $(u,v)$ a corresponding word over $(\Xd)^2$ whose $j$-th letter $(x_j,y_j)$ for $1\leq j\leq n$.

By definition of the convolution we have $L_{pairs}=L_1\cap L_2\cap L_3$, where
\[\begin{array}{lll}
L_1&=&\{(u,v)\in ((\Xd)^2)^*\mid u\in L\diamond^*, v\in \Xd^*\},\\
L_2&=&\{(u,v)\in ((\Xd)^2)^*\mid u\in\Xd^*, v\in L\diamond^*\},\\
L_3&=&\{(u,v)\in ((\Xd)^2)^*\mid (u,v)\ \text{has no letter}\ (\diamond,\diamond)\\
&&\text{and has no subwords}\ (x,\diamond)(y,z)\ \text{and}\ (\diamond,x)(z,y), z\in X, x,y\in\Xd\}.
\end{array}\]

The languages $L_1$ and $L_2$ are regular. We can build an automaton $\A_{L_1}$ over $\Xd^2$ recognizing $L_1$ from the automaton $\A_L$ over $X$ with the state set $Q$ recognizing $L$ as follows. The set of states of $\A_{L_1}$ is $Q\cup\{t\}$, where $t$ is a terminal state of $\A_{L_1}$ not in $Q$. For each transition $q_1\stackrel{x}{\longrightarrow}q_2$ in $\A_L$ for $q_1,q_2\in Q$ and $x\in X$, we introduce $|X|+1$ transitions of the form $q_1\stackrel{(x,y)}{\longrightarrow}q_2$, $y\in\Xd$. Additionally, for each terminal state $t'\in Q$ we introduce $|X|+1$ transitions of the form $t'\stackrel{(\diamond,y)}{\longrightarrow}t$, $y\in\Xd$. The automaton recognizing $L_2$ is constructed similarly.

Since the language $L_3$ is clearly regular over $(\Xd)^2$, we get that $L_{pairs}$ is regular as the intersection of three regular languages.
\end{proof}

\begin{lemma}
\label{lem:edges}
For each $s\in\{a,b\}$ the binary relation $\overline{E}_s$ is regular over $X$.
\end{lemma}

\begin{proof}
We have to prove that $\otimes\overline{E}_a$ and $\otimes\overline{E}_b$ are regular languages over $(\Xd)^2$. We will show that both of these languages can be obtained as an intersection of a regular language $L_{pairs}$ with regular languages accepted by automata built by modifying the automaton $\A$ generating the group $\G$.

\begin{figure}
\begin{center}
\epsfig{file=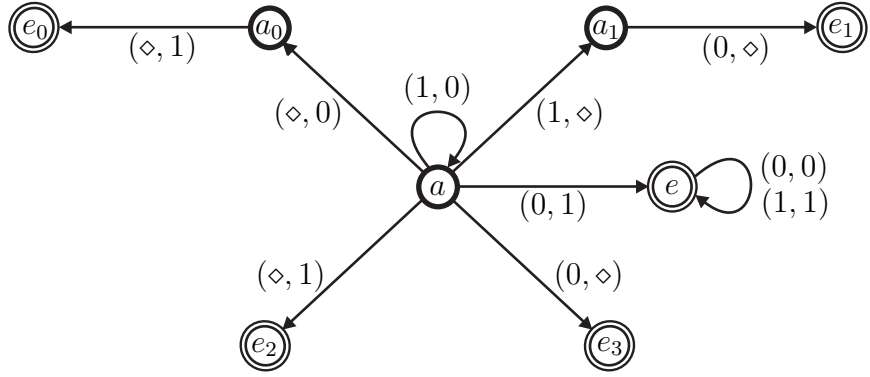}
\end{center}
\caption{Automaton $\A_a$ accepting $L_a$\label{fig:autom_a}}
\end{figure}

We start from $\otimes\overline{E}_a$. Let $L_a$ be a regular language over $(\Xd)^2$ recognized by an automaton $\A_a$ shown in Figure~\ref{fig:autom_a}, where the initial state is $a$ and terminal states are marked by double circles. Of course, states $e_0$ through $e_3$ are equivalent, but we intentionally separate them to make the connection between automata $\A$ and $\A_a$ more clear and to emphasize different cases in the proof. We will show that
\begin{equation}
\label{eqn:E_a}
\otimes\overline{E}_a=L_a\cap L_{pairs},
\end{equation}
thus proving that $\otimes\overline{E}_a$ is regular language over $(\Xd)^2$.

Recall that the vertices of $\Gamma$ are labelled by infinite words over $X$ cofinal with $(01)^\infty$. We will identify the vertices with their labels. By definition of $\overline{\phantom{V}}\colon V(\Gamma)\to X^*$ the preimage of $v\in\overline{V(\Gamma)}$ under $\overline{\phantom{V}}$ is the vertex $\xi_v=x_1x_2x_3\ldots$, where
\[x_i=\left\{\begin{array}{ll}
\text{$i$-th letter of $v$}, & \text{if $i\leq|v|$,}\\
\text{$i$-th letter of $(01)^\infty$}, & \text{if $i>|v|$.}
\end{array}\right.\]

Suppose $(u,v)\in\otimes\overline{E}_a$ for some $u,v\in\Xd^*$ of the same length. Then there are words $u',v'\in\overline{V(\Gamma)}$ such that $(u,v)=\otimes(u',v')$ and such that there is an edge in $\Gamma$ labelled by $a$ from $\xi_{u'}$ to $\xi_{v'}$. Since $(u,v)=\otimes(u',v')$ by the definition of $L_{pairs}$ we immediately get that $(u,v)\in L_{pairs}$.

We will now show that $(u,v)$ is in $L_a$, i.e. accepted by $\A_a$. By definition of the adjacency in $\Gamma$, we can read an infinite sequence of pairs of letters in $X$ corresponding to the pair $(\xi_{u'},\xi_{v'})$ of infinite words over $X$ by following the transitions in automaton $\tilde A$ over $X^2$ depicted in Figure~\ref{fig:929acceptor} starting from state $a$.

\begin{figure}
\begin{center}
\epsfig{file=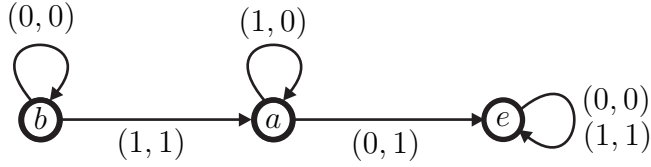}
\end{center}
\caption{Modified automaton $\tilde\A$\label{fig:929acceptor}}
\end{figure}

\begin{figure}
\begin{center}
\epsfig{file=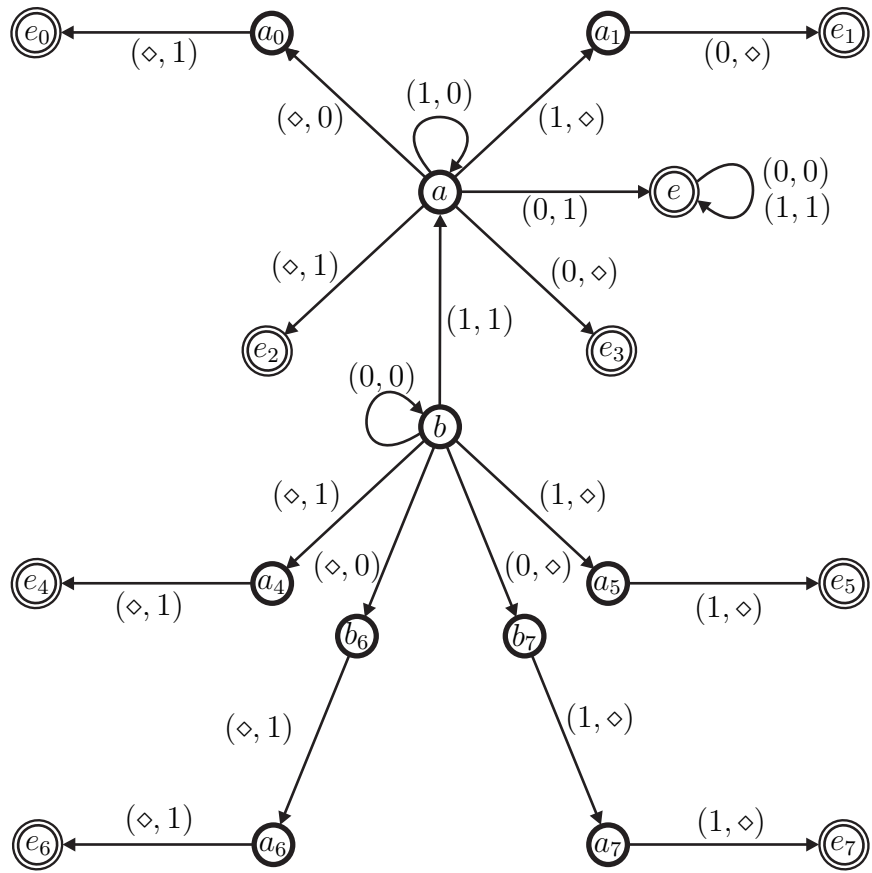}
\end{center}
\caption{Automaton $\A_b$ accepting $L_b$\label{fig:autom_b}}
\end{figure}

Suppose first, that both $u=x_1x_2\ldots x_k$ and $v=y_1y_2\ldots y_k$ do not contain $\diamond$, and, hence, $u'=u$ and $v'=v$. In this case we can disregard all transitions in $\A_a$ with labels containing $\diamond$. After removing all such transitions and corresponding states from $\A_a$ we get an automaton equivalent to automaton $\tilde\A$ with initial state $a$. In particular, we read the same words over $X^2$ along paths in these automata. Therefore, we can read $(u,v)$ along the path in $\A_a$. The question is only if we end up in the accepting state of $\A_a$.

Since the pair consisting of two empty words is not accepted by $\A_a$ and is not in $\otimes\overline{E}_a$, we can assume that $k>0$. Observe that both $x_k$ and $y_k$ must be different from the $k$-th letter in $(01)^\infty$ since otherwise $u$ or $v$ would not be in $\overline{V(\Gamma)}$. In particular, we get that $x_k=y_k$. But this means that we have to be in the accepting state $e$ of $\A_a$ after reading $(u,v)$ starting from state $a$. Therefore, $(u,v)$ is accepted by $\A_a$.

Now assume that $|u|=|u'|>|v'|$. Then $v$ has a form $y_1y_2\ldots y_l\diamond^{k-l}$ for some $0\leq l<k$, and the pair of infinite words $(\xi_u,\xi_{v'})$ can be read along the path in $\tilde A$.

Note that by definition of $\overline{\phantom{V}}\colon V(\Gamma)\to X^*$  the letter $x_{k}$ of $u$ is different from the $k$-th symbol of $(01)^\infty$, which coincides with the $k$-th symbol of $\xi_{v'}$. Therefore, while reading a pair containing $x_k$ by $\tilde\A$ with initial state $a$ we must be at state $a$. But since as soon as we leave state $a$ we never come back, it follows that we must remain in state $a$ after reading $k$ first pairs of letters in $(\xi_u,\xi_{v'})$. Consequently, we will be in the state $a$ after reading $(x_1x_2\ldots x_l,y_1y_2\ldots y_l)$ along the path in $\A_a$.

After reading the first $l$ pairs of letters of $(u,v)$ by $\A_a$, the next pair we read is $(x_{l+1},\diamond)$. Consider two cases:
\begin{enumerate}
\item If $x_{l+1}=0$, then since we are in the state $a$ in $\tilde\A$ after reading first $l$ pairs, we shift to the state $e$ in $\tilde A$ after reading the pair containing $x_{l+1}$. After this, the automaton $\tilde\A$ will accept only pairs of identical letters. In particular, we get that $\xi_{u}$ coincides $\xi_v$, and thus with $(01)^\infty$ by construction of $\xi_v$, at positions starting from $l+2$. Therefore, $x_{l+1}$ must be different from the $(l+1)$-st letter in $(01)^\infty$. Hence, $u=x_1x_2\ldots x_l0$, and after reading $(u,v)=(x_1x_2\ldots x_l0,y_1y_2\ldots y_l\diamond)$ starting from the state $a$ the automaton $\A_a$ will shift to the accepting state $e_3$ and will accept $(u,v)$.

\item If $x_{l+1}=1$, then while reading the $(l+1)$-st letter of $(u,v)$ by automaton $\tilde\A$ we must stay at the state $a$, because we have to follow the arrow, the first coordinate of whose label is 1. But as there is just one arrow whose label has the first coordinate 1 going out of state $a$, this determines uniquely the second coordinate of this label, which must be 0. Therefore, the $(l+1)$-st letter of $\xi_{v'}$, and thus of $(01)^\infty$ is 0. But this implies that the $(l+2)$-nd letter in $(01)^\infty$ is $1$. Note that this is precisely the only place where we need that $\omega=(01)^\infty$, because we need the next letter of $\omega$ to be completely determined by the previous one, so we have to choose $\omega$ from $0^\infty$, $1^\infty$, $(01)^\infty$ and $(10)^\infty$.

Now if $(l+2)$-nd letter of $(01)^\infty$, and thus of $\xi_{v'}$, is 1, in the automaton $\tilde A$ we have to follow the arrow going out of a state $a$, whose label's second coordinate is $1$. There is again exactly one such arrow, that ends up in the state $e$ and whose label is $(0,1)$. Thus, $x_{l+2}=0$, and $\xi_u$ and $(01)^\infty$ coincide at positions $l+3$ and higher. Therefore, $u=x_1x_2\ldots x_{l}10$, and after reading $(u,v)=(x_1x_2\ldots x_l10,y_1y_2\ldots y_l\diamond\diamond)$ starting from the state $a$ the automaton $\A_a$ will shift to the accepting state $e_1$ and will accept $(u,v)$.
\end{enumerate}

The case when $|v|=|v'|>|u'|$ is analogous. In this case after reading $(u,v)$ the automaton $\A_a$ will end up either in the state $e_0$ or in the state $e_2$.
Therefore, each word in $\otimes\overline{E}_a$ is accepted by $A_a$.

Conversely, if a word $(u,v)$ over $(\Xd)^2$ is accepted by $\A_a$, then after reading this word the automaton has to shift to one of the five terminal states. Consider all cases separately:
\begin{enumerate}
\item If we end up in the state $e_0$, then $(u,v)=(1^n\diamond\diamond,0^n01)$ for some $n\geq0$. In the case $n$ is even, the word $0^n01$ is not in $\overline{V(\Gamma)}$, and thus $(u,v)\notin L_{pairs}$. If $n$ is odd, then $(u,v)=\otimes(1^n,0^n01)$, where both $1^n$ and $0^n01$ are elements of $\overline{V(\Gamma)}$ corresponding to vertices \[\xi_u=1^n1(01)^\infty\]
and
\[\xi_v=0^n011(01)^\infty\]
that are connected by the edge in $\Gamma$ labelled by $a$, because one can read $(\xi_u,\xi_v)$ along the path in the automaton $\tilde\A$.
\item The cases when we end up in states $e_1$, $e_2$ and $e_3$ are treated in the same way.
\item If we end up in the state $e$, then both $u$ and $v$ do not contain $\diamond$. We will be able to read $(\xi_u,\xi_v)$ along the path in the automaton $\tilde A$. So there is an edge from $\xi_u$ to $\xi_v$ in $\Gamma$. Therefore, in this situation, $(u,v)\in L_{pairs}$, if and only if $(u,v)\in\otimes\overline{E}_a$.
\end{enumerate}

Thus, each word that is accepted by $\A_a$ and is in $L_{pairs}$ must be in $\otimes\overline{E}_a$. This finishes the proof of the equality~\eqref{eqn:E_a}.

Similarly to $L_a$ we define $L_b$ to be a regular language recognized by automaton $\A_b$ depicted in Figure~\ref{fig:autom_b}. Similarly to the definition of an automaton $\A_a$, states $e_0$ through $e_7$ are equivalent, but we intentionally separate them to make the diagram of an automaton more clear. The proof that
\begin{equation*}
\otimes\overline{E}_b=L_b\cap L_{pairs},
\end{equation*}
is analogous to the proof of equality~\eqref{eqn:E_a}.
\end{proof}

\begin{proof}[Proof of Theorem~\ref{thm:main}]
The main Theorem~\ref{thm:main} now follows by the definition of an automatic graph from Lemma~\ref{lem:vertices} and Lemma~\ref{lem:edges}.
\end{proof}
\bigskip

\noindent{\textbf{\large References}
\bigskip

\bibliographystyle{plain}


\def\cprime{$'$} \def\cprime{$'$} \def\cprime{$'$} \def\cprime{$'$}
  \def\cprime{$'$} \def\cprime{$'$} \def\cprime{$'$} \def\cprime{$'$}
  \def\cprime{$'$} \def\cprime{$'$} \def\cprime{$'$} \def\cprime{$'$}
  \def\cprime{$'$}

\end{document}